\documentclass[11pt,a4paper]{article}
\title{Set estimation when sampling from random lines}
\author{A. Cholaquidis, A. Cuevas and L. Moreno}
\usepackage[]{geometry}
\usepackage[mathcal]{euscript}
\usepackage{subcaption}
\usepackage{bm,url}                     
\usepackage{amsfonts}
\usepackage{amssymb}
\usepackage{amsmath}
\usepackage{amsthm}
\usepackage{tipa}
\usepackage[utf8]{inputenc}
\usepackage{amsmath,amsthm}
\usepackage{graphicx}
\usepackage{hyperref}
\usepackage{color}
\usepackage[normalem]{ulem}
\usepackage[dvipsnames]{xcolor}
\definecolor{bu}{rgb}{0.5, 0.0, 0.13}
\usepackage{todonotes}             
\usepackage{natbib}                 
\usepackage{colortbl}               
\usepackage{booktabs}

\newtheorem{lemma}{Lemma}
\newtheorem{definition}{Definition}
\newtheorem{theorem}{Theorem}

\newtheorem{proposition}{Proposition}
\newtheorem{remark}{Remark}

\definecolor{violet}{rgb}{0.7,0.2,0.6}

\newcommand{\eps}{\varepsilon}

\begin{document}

	\begin{center}
		\Large \bf On the notion of polynomial reach: \\ a statistical application
	\end{center}
	
	\begin{center}
		Alejandro Cholaquidis$^1$, Antonio Cuevas$^2$ and Leonardo Moreno$^3$\\
		\footnotesize 	$^1$ Centro de Matem\'atica, Universidad de la República, Uruguay\\
		$^2$ Departamento de Matem\'aticas, Universidad Aut\'onoma de Madrid\\ and Instituto de Ciencias Matem\'aticas ICMAT (CSIC-UAM-UCM-UC3M)\\
		$^3$ Instituto de Estadística, Universidad de la República, Uruguay
	\end{center}

\begin{abstract}
The volume function $V(t)$ of a compact set $S\in{\mathbb R}^d$ is just the Lebesgue measure of the set of points within a distance to $S$ not larger than $t$. According to some classical results in geometric measure theory, the volume function turns out to be a polynomial, at least in a finite interval, under a quite intuitive, easy to interpret, sufficient condition (called ``positive reach'')  which can be seen as an extension of the notion of convexity.
However, many other simple sets, not fulfilling the positive reach condition, have also a polynomial volume function. To our knowledge, there is no general, simple geometric description of such sets.  Still, the polynomial character of $V(t)$ has some relevant consequences since the polynomial coefficients carry some useful geometric information. In particular, the constant term is the volume of $S$ and the first order coefficient is the boundary measure (in Minkowski's sense). This paper is focused on sets whose volume function is polynomial on some interval starting at zero, whose length (that we call ``polynomial reach'') might be unknown. Our main goal is to approximate such polynomial reach by statistical means, using only a large enough random sample of points inside $S$. The practical motivation is simple: when the value of the polynomial reach, or rather a lower bound for it, is approximately known, the polynomial coefficients can be estimated from the sample points by using standard methods in polynomial approximation. As a result, we get a quite general method to estimate the volume and boundary measure of the set, relying only on an inner sample of points and not requiring the use any smoothing parameter. This paper explores the theoretical and practical aspects of this idea. 
\end{abstract}

\section{Introduction}\label{sec:intro}

Given a compact set $S\subset {\mathbb R}^d$ and $r>0$, the $r$-parallel set of $S$, denoted here by $B(S,r)$, is the set of all points in ${\mathbb R}^d$ for which there is a point of $S$ within a distance not larger than $r$; see below in this section for formal definitions. 

The ``volume function'' of $S$ is then defined by $V(r)=\mu(B(S,r))$, where $\mu$ denotes the Lebesgue measure on ${\mathbb R}^d$. As it happens, this function carries a lot of useful information on the geometry of $S$. In particular, $V(0)=\mu(S)$. Also, the limit, when it does exist,
\begin{equation}\label{Mink}
L(S)=\lim_{\varepsilon\to 0^+}\frac{\mu(B(S,\varepsilon)\setminus S)}{\varepsilon}=\lim_{\varepsilon\to 0^+}\frac{V(\varepsilon) - V(0)}{\varepsilon},
\end{equation}
 provides a natural way of defining the surface measure of $S$. The value $L(S)$ is called the \textit{outer Minkowski content} of $\partial S$; see Ambrosio et al. (2008) for a detailed study of this and other notions of surface measure. 

In many important cases the function $V$ is a polynomial of degree at most $d$ on some interval $[0,R]$. The best known example is given by the so-called ``sets of positive reach'' introduced in a celebrated paper by \cite{fed59}. The reach of a compact set $S$ is the supremum of the values $\mathbf{r}$ such that any point outside $S$ has only one metric projection on $S$; see below for more details. 

Of course, when the volume function $V$ of $S$ is a polynomial, the constant term is $V(0)=\mu(S)$ and the coefficient of the first-order term is $V'(0)=L(S)$. \\

\noindent \it The purpose of this work\rm 

The general aim of this work is to exploit the above mentioned polynomial assumption for statistical purposes. We follow the lines of  \cite{cuepat18} where it is shown that $V(r)$ can be consistently estimated from a random sample of points inside $S$. Then, if we denote $V_n(r)$ an estimator of $V$ based on a sample of size $n$, the coefficients $V(r)$ can be estimated by a minimal distance procedure, just approximating $V_n$ from the closest polynomial of degree $d$ on the interval $[0,R]$ of validity of the polynomial assumption.

The present paper addresses two topics in this framework. First and foremost, we consider (from both the theoretical and practical  point of view) the estimation of the ``polynomial reach'' ${\mathbf R}$, that is, the maximum value of $R$  for which the polynomial assumption holds on $[0,R]$. This allows us, as an important by-product, to address the statistical estimation of $L(S)$ and $\mu(S)$ from a random sample of points, which remains as the primary motivation of the whole study. To be more precise, if  our main goal is to estimate the coefficients of  the polynomial volume we do not need in fact to estimate the polynomial reach ${\mathbf R}$. A infra-estimation of this parameter would be enough, and safer than a possible over estimation which might lead to an erroneous polynomial fit for the volume function. 

The word ``reach'' is used here by analogy with the ordinary, geometric notion of Federer's reach, mentioned above. Such analogy is motivated by the fact that, as proved by Federer (1959), if the reach ${\mathbf r}$ of a compact set is positive, then the volume function  is a polynomial on $[0,{\mathbf r}]$. However this sufficient condition is by no means necessary, since many simple sets with ${\mathbf r}=0$ have a polynomial volume on some interval. So if we are just interested in the polynomial volume property  we could perhaps focus on ${\mathbf R}$ rather than ${\mathbf r}$.

We will comment at some more detail these statistical aspects in Section \ref{sec:perspective}. However, let us now advance that 
\begin{itemize}
	\item[(a)] we will assume that our sample information comes just from an inside sample on $S$, unlike other approaches that also require sample information outside $S$; see Section \ref{sec:perspective} for  details and references.
	\item[(b)] Our proposal does not require to estimate the set $S$ itself as a preliminary step. 
	\item[(c)]  The conditions imposed on $S$ are lighter than others appearing in the literature.
\end{itemize}

\noindent \it Some notation and preliminary definitions\rm

Given a set $S\subset \mathbb{R}^d$, we will denote by
$\mathring{S}$ and $\partial S$ the interior and boundary of $S$,
respectively with respect to the usual topology of $\mathbb{R}^d$. 

The parallel set of $S$ of radius $\varepsilon$ will be denoted as $B(S,\eps)$, that is
$B(S,\eps)=\{y\in{\mathbb R}^d:\ \inf_{x\in S}\Vert y-x\Vert\leq \eps \}$.
If $A\subset\mathbb{R}^d$ is a Borel set, then $\mu_d(A)$ (sometimes just $\mu(A)$) will denote 
its Lebesgue measure.
We will denote by $B(x,\varepsilon)$  the closed ball
in $\mathbb{R}^d$,
of radius $\varepsilon$, centred at $x$, and $\omega_d=\mu_d(B(x,1))$.

Given two compact non-empty sets $A, C \subset{\mathbb R}^d$, 
the \it Hausdorff distance\/ \rm or \it Hausdorff-Pompei distance\/ \rm between $A$ and $C$ is defined by
\begin{equation}
	d_H(A,C)=\inf\{\eps>0: \mbox{such that } A\subset B(C,\eps)\, \mbox{ and }
	C\subset B(A,\eps)\}.\label{Hausdorff}
\end{equation}

If $I$ is an interval in ${\mathbb R}$ denote by $L^2(I)$ the space of real square-integrable functions defined on $I$, endowed with the usual norm, $\Vert f\Vert_{L^2(I)}=\sqrt{\int_I f^2(s)ds}$.

Let us denote $V(r)= \mu(B(S,r))$, for $r\geq 0$, the volume function of the set $S$. If $\aleph_n=\{X_1,\ldots,X_n\}$ stands for a sample of points $X_i$ on $S$ we will denote by  $$V_n(r)=\mu(B(\aleph_n,r))$$ 
the empirical volume function. 

Given $r>0$ and a closed interval $I$ in $[0,\infty)$, we denote by $P_{n,\ell}^I$  and $P_\ell^I$, respectively, the best approximations, by polynomials of degree at most $\ell\geq d$, of $V_n$ and $V$, with respect to  the $L^2$ norm. That is, if $\Pi_\ell(I)$ denotes the (closed) subspace of all polynomials of degree at most $\ell\in \mathbb{N}$ in $L^2(I)$,
\begin{equation}\label{eq2}
	P_{n,\ell}^I=\text{argmin}_{\pi \in \Pi_\ell(I)} \Vert V_n-\pi\Vert_{L^2(I)}, \quad \text{and}\quad  P_\ell^I=\text{argmin}_{\pi \in \Pi_\ell(I)} \Vert V-\pi \Vert_{L^2(I)}.
\end{equation}

As indicated below, we are particularly interested on $P_{n,d}^{I}$, $d$ being the dimension of our data points. Let us denote  $P_{n,d}^{I}(t)=\theta_{0n}+\dots+\theta_{dn}t^d$ for $t\in I\subset [0,\mathbf{R}]$.
In practice, the computation of $P_{n,d}^{I}$ from \eqref{eq2} is done numerically in a simple way: the values $\theta_{jn}$ are just fitted as the coefficients of a linear regression model where the response variables are values $V_n(r_i)$, calculated, by simulation, with arbitrary precision, for a given grid of points $r_1<\cdots<r_N$ in $I$.  

Following the notation in \cite{fed59}, let ${\rm \text{Unp}}(S)$ be the set of points $x\in \mathbb{R}^d$ 
with a unique metric projection on $S$.

For $x\in S$, let \mbox{reach}$(S,x)=\sup\{r>0:\mathring{B}(x,r)\subset {\mbox{Unp}}(S)\big\}$. 
The \textit{reach of $S$} is then defined by $\mbox{reach}(S)=\inf\big\{\mbox{reach}(S,x):x \in S\big\}$, and $S$
is said to be of positive reach if ${\mathbf r}:=\mbox{reach}(S)>0$. 

\

A set $S\subset \mathbb{R}^d$ is said to be \textit{standard} with respect to a
Borel measure $\nu$ if there exists $\lambda>0$ and $\delta>0$ such that, for all $x\in S$,
\begin{equation} \label{estandar}
	\nu(B(x,\eps)\cap S)\geq \delta \mu_d(B(x,\eps)),\quad 0<\eps\leq \lambda.
\end{equation}

 Informally speaking, this condition prevents the set $S$ from being ``too spiky'' with respect to the measure $\nu$. See, e.g., \cite{cue04} and references therein for details on the use of the standardness condition in set estimation. 

\

In addition to the \textit{outer Minkowski content} \eqref{Mink}, an alternative way of measuring  is the two-side version, simply known as \textit{Minkowski content},
\begin{equation}\label{Mink-two}
	L_0(S)=\lim_{\varepsilon\to 0^+}\frac{\mu(B(\partial S,\varepsilon))}{2\varepsilon}.
\end{equation}	
The relation of this notion with its one-sided version \eqref{Mink} and with the, perhaps more popular, concept of $(d-1)$-dimensional Hausdorff measure 
${\mathcal H}^{d-1}(\partial S)$ (that will be mentioned below in the proof of Lemma 1) is analyzed in \cite{amb08}.

\

\noindent \textit{Organization of this work}

In the following section  the notion of polynomial reach is formally introduced. Also, some perspective and motivation are given in order to show the usefulness of such notion.  The estimation of the polynomial reach from a random sample of points is considered in Section 3. Two methods are proposed: one of them is asymptotically consistent to the true value of the reach, but not that useful in practice; the other one provides a infra-estimation, which is enough for most practical purposes. Convergence rates for the estimation of the polynomial coefficients are derived in Section 4. Some numerical experiments are commented in Section 5. Finally, a few conclusions are briefly commented in Section 6, a few technical proofs are included in Appendix A and some tables with numerical outputs are provided in Appendix B.

\section{Some perspective and motivation. The notion of polynomial reach}\label{sec:perspective}

As mentioned above, our aim here is exploiting a geometric idea to address some statistical problems in the setup of set estimation. The general purpose of this theory is to reconstruct a (compact) set $S$ from the information provided by a random sample of points. A brief survey can be found in \cite{cue09}. See, e.g., \cite{chola:14} and \cite{aar17} for more recent references, including connections to manifold learning and other relevant topics. 

In many cases one is mostly interested in estimating a functional 
of $S$, typically the Lebesgue measure $\mu(S)$ or the outer Minkowski content (boundary measure) $L(S)$ as defined in \eqref{Mink}.
Such problems have been addressed in the literature from different strategies, which we next summarize.
\begin{itemize}
	\item[A)]  \textit{Plug-in approaches,} based on a shape assumption on $S$. For example, if $S$ is assumed to be convex it would be quite natural to use the volume or the boundary measure of the convex hull of the sample $\aleph_n=\{X_1,\ldots,X_n\}$ as an estimator of the  values $\mu(S)$ and $L(S)$, respectively. See, e.g., \cite{bal21} for a recent reference on the plug-in estimation of $\mu(S)$ under convexity. In \cite{cue12} and \cite{ari19} the analogous plug-in estimation of $L(S)$ and $\mu(S)$ under the wider assumption of $r$-convexity is considered.   In this case, the plug-in estimators of  $L(S)$ and $\mu(S)$ would be $L(S_n)$ and $\mu(S_n)$, $S_n$ being the $r$-convex hull of the sample 
	\item[B)]  \textit{Methods based on two-samples.} In some cases, one may assume that one has two samples, one inside and the other outside $S$. This extra information might allow for estimators of $L(S)$ essentially based on nearest neighbors ideas; see, for example,  \cite{cue07} and \cite{jim11}.

	\item[C)]  \textit{Indirect methods,} based on auxiliary functions or formulas involving the surface area. This is the case of \cite{cuepat18} or \cite{aar22}. Also the results on the asymptotic distribution of the Hausdorff distance between a random sample and its support provided by \cite{pen23} are of potential interest in this regard.
\end{itemize}
The present paper fits in the item C) of this list. More specifically we follow the lines of \cite{cuepat18}. However, whereas in that paper the interval of validity $[0, \textbf{R}]$ of the polynomial assumption is assumed to be known, we consider here the non-trivial problem of estimating such interval. The motivation for this is to use the polynomial character of $V(t)$ in that interval to estimate the polynomial coefficients which, as commented above, have a relevant geometric interest. This can be seen as a sort of ``algebraic counterpart'' of the estimation of Federer's reach parameter defined above. Nevertheless, it is important to note that the assumption of positive reach (which is very relevant in many other aspects) is not needed or used here at all. Of course, if we assume that $S$ has a positive reach ${\mathbf r}>0$, the polynomial volume assumption would be ensured in the interval $[0,{\mathbf r}]$. Therefore, any method to estimate Federer's reach ${\mathbf r}$ (see, e.g., \cite{chola22}) would be useful  to exploit the polynomial volume assumption.  The point is that there are many extremely simple sets for which Federer's reach is 0 and, still, polynomial volume assumption does hold. These include a ``pacman-type'' set such as the closed unit disk in ${\mathbb R}^2$ excluding an open sector, the union of two disjoint squares, or a simple set such as $[-1,1]^2\setminus (-\frac{1}{2},\frac{1}{2})^2$. In all these cases, and in many others, the following definition applies 
\begin{definition}\label{def:pvol}
	A compact set $S\subset{\mathbb R}^d$ is said to fulfil the polynomial volume property if there exist  constants  $\theta_0,\ldots,\theta_d\in{\mathbb R}$ and $R>0$ such that
	\begin{equation} 
		V(r)=\theta_0+\theta_1r+\ldots+\theta_dr^d,\ \mbox{for all } r\in[0,R].\label{pvol}
	\end{equation}	
	
\end{definition}

When this condition holds, a natural strategy to follow is to estimate $V(r)$ by its empirical counterpart $V_n(r)$, as defined in the previous section and, in turn,  to approximate $V_n(r)$ by a polynomial of degree $d$, whose coefficients $\theta_{0n},$ and $\theta_{1n}$ can be seen as estimators of $\mu(S)$ and $L(S)$, respectively. 

The above definition leads in a natural way to the following notion of polynomial reach, which is the central concept of the present work.

\begin{definition}
	Given a compact set $S\in {\mathbb R}^d$ with volume function $V$, we will define the \textit{polynomial reach} ${\mathbf R}$ of $S$ as 
	\begin{equation}\label{pol_reach}
		{\mathbf R}=\sup\{r\geq 0:\ V \mbox{ is a polynomial of degree at most }\ d \mbox{ on }\ [0,r]\}.
	\end{equation}
\end{definition}

 When the set \( S \) possesses positive reach (as defined above in the Section \ref{sec:intro}), it is established in \cite{fed59} that in \eqref{pvol}, \( \theta_0 \) equals \( \mu_d(S) \), \( \theta_1 \) is \( L(S) \), \( \theta_2 \) represents the integrated mean curvature, and \( \theta_d \) is \( \omega_d\chi(S) \), where \( \chi(S) \) denotes the Euler-Poincaré characteristic of \( S \).  These geometric interpretations for the polynomial coefficients still hold true for $\theta_0$ and $\theta_1$ when we just assume positive polynomial reach. However, they do not necessarily apply for the other coefficients related to curvatures and Euler's characteristic. 
On the other hand, the relation between Federer's reach $\mathbf r$ and polynomial reach $\mathbf R$  is not straightforward. Of course $\mathbf r\leq \mathbf R$, but it is not obvious when $\mathbf r=\mathbf R$; this is the case, for example, when $S$ is te union of two points. 
See \cite{berr} for additional details and examples on the positive polynomial reach condition.

\section{Consistent estimation of the polynomial reach}

The aim of this section and in fact, the main theoretical contribution of this work, is to show that, under some conditions, one can obtain either a consistent estimator (subsection \ref{subsec:consistent})
or, asymptotically, a lower bound (subsection \ref{subsec:algorithm})  of the polynomial reach ${\mathbf R}$. Note that, somewhat paradoxically, a lower bound might be even preferable and ``safer'' since, in order to use the polynomial volume assumption to estimate the relevant geometric quantities (area, perimeter,...) all we need is an interval where the polynomial expression does hold. Thus, we do not need in fact the whole interval $[0,{\mathbf R}]$: we may more easily afford some underestimation of ${\mathbf R}$ rather than an error by excess in the estimation of this quantity.

Let us start with two technical lemmas with some independent interest, whose proofs are given in Appendix A.

\begin{lemma} \label{lemaux2} Let $S\subset\mathbb{R}^d$ be a compact set such that   $V(0)>0$ and the right-hand side derivative $V^+(0)$
	does exists and is finite.  Let $\aleph_n=\{X_1,\ldots,X_n\}\subset S$ be  a sequence of finite sets  such that $\gamma_n=d_H(\aleph_n,S)\to 0$ and $2\gamma_n<b<+\infty$. Then,  for all $n$ large enough,  
\begin{equation}\label{convunif}
\sup_{s\in [2\gamma_n,b]}|V(s)-V_n(s)| \leq C\max_{p\in \partial S}d(p,\aleph_n), \mbox{ where } C=\mbox{\rm ess sup}_{s\in [0,b]} |V'(s)|,
\end{equation}
 and 
\begin{align}\label{bound0gamman}
\Vert V-V_n\Vert_{L^2([0,2\gamma_n])}\leq   2V(0)(2\gamma_n)^{1/2}.
\end{align}
As a consequence,  given the compact interval $I=[0,b]$ we have that, for all $n$ large enough,
\begin{equation}\label{2gamman}
\Vert V-V_n\Vert_{L^2(I)}\leq 2\sqrt{2}V(0)\gamma_n^{1/2}+\sqrt{b}C\max_{p\in \partial S}d(p,\aleph_n).
\end{equation}

\end{lemma}

\begin{lemma}\label{rates1} Let $\aleph_n=\{X_1,\ldots,X_n\}$ be an iid sample on a compact set $S$, that comes from a distribution $P_X$ standard with respect to Lebesgue measure.  Assume further that  $S$ fulfils $L_0(S)<\infty$.
	
	Then with probability one,  
	\begin{equation} \label{bounddhborde}
		\limsup_{n\to \infty }  \Big(\frac{n}{\log n}\Big)^{1/d}\max_{p\in \partial S}d(p,\aleph_n)\leq \Big(\frac{2}{\delta\omega_d}\Big)^{1/d},
	\end{equation}
	being $\delta$ the standardness constant of $P_X$.
\end{lemma}

\subsection{A consistent estimator}\label{subsec:consistent}

The main focus of this subsection is to show
in Proposition \ref{prop:consist}
that the polynomial reach can be estimated consistently from a sample. While this result has some conceptual interest, it suffers from some practical limitations: on the one hand, the estimator we propose might have (as suggested by some numerical experiments) a rather poor accuracy except for very large sample sizes. In the second place it could provide in some cases an overestimation which could entail some practical drawbacks commented at the beginning of the following subsection. 

Let us denote by simplicity $P_n^{t}:=P_{n,d}^{[0,t]}$, $P^t:=P_d^{[0,t]}$ the best polynomial approximation (in $L^2$)  of degree at most $d$ for $V_n$ and $V$, respectively, on $[0,t]$. Define
$$
G_n(t)=\Vert V_n-P_n^t\Vert_t=\Vert V_n-P_n^t\Vert_{L^2[0,t]},\ \text{ and }\  G(t)=\Vert V-P^t\Vert_t:=\Vert V-P^t\Vert_{L^2[0,t]}.
$$

Note that $G_n(t)$ is the polynomial approximation error for the estimate volume $V_n$ and $G(t)$ is the polynomial approximation error for the true volume $V$. Then $G(t)=0$ whenever $t\leq \mathbf{R}$.  Also, note that  $\mathbf{R}=\lim_{\varepsilon\to 0} \inf\{t:G(t)>\varepsilon\}$.

\begin{theorem}\label{prop:consist}
	Let $S\subset \mathbb{R}^d$  be compact  and $\aleph_n=\{X_1,\dots,X_n\}\subset S $ a sequence of finite sets such that $\gamma_n:=d_H(\aleph_n,S)\to 0$.
	Take a sequence $\varepsilon_n>0$ such that  $\varepsilon_n\to 0$ and 
	$$
	\gamma_n^{1/2}=o(\varepsilon_n),\ \mbox{ as } n\to\infty, $$
	then $\tilde{\mathbf{R}}=G_n^{-1}(\varepsilon_n)=\inf\{t: G_n(t)>\varepsilon_n\}$ fulfils
	\begin{equation} \label{limth}
		\tilde{\mathbf{R}}\to {\mathbf R}.
	\end{equation}
	\begin{proof}

		Let us denote $\pi^t$ the orthogonal $L^2$ projection onto the  space of polinomials of degree at most $d$ on $[0,t]$. Let $b>0$, and $n$ large enough such that $2\gamma_n<b$. Then, using the contractive property for the projections on convex sets in Hilbert spaces, we have 
		\begin{multline*}
			\sup_{0<2\gamma_n<t<b}\Vert P_n^t-P^t\Vert_t=\sup_{0<2\gamma_n<t<b}\Vert\pi^t(V_n)-\pi^t(V)\Vert_t\leq \sup_{0<2\gamma_n<t<b}\Vert V_n-V\Vert_t \\
			\leq \Vert V_n-V\Vert_{2\gamma_n}+	 \sqrt{(b-2\gamma_n)}\sup_{0<2\gamma_n<t<b}\sup_{s\in [2\gamma_n,t]} |V_n(s)-V(s)|.
		\end{multline*}
		From \eqref{bound0gamman} $\Vert V_n-V\Vert_{2\gamma_n}\leq  2V(0)(2\gamma_n)^{1/2}$.
		Using  \eqref{convunif} and $\max_{p\in \partial S}d(p,\aleph_n)\leq \gamma_n \to 0$, it follows that
		$$\sqrt{(b-2\gamma_n)}\sup_{0<2\gamma_n<t<b}\sup_{s\in [2\gamma_n,t]} |V_n(s)-V(s)|\to 0$$
		Then  $\sup_{0<2\gamma_n<t<b}\Vert P_n^t-P^t\Vert_t\to 0$ as $n\to \infty$. 
		From the triangular inequality  we have that, with probability one,
		
		\begin{equation} \label{cota1}
			\sup_{0<2\gamma_n<t<b} \Big|\Vert V_n-P_n^t\Vert_t- \Vert V-P^t\Vert_t\Big| \leq \sup_{0<2\gamma_n<t<b} \left( \Vert V_n-V\Vert_t+\Vert P^t-P_n^t\Vert_t \right) =O(\gamma_n^{1/2}),
		\end{equation}
		\noindent
		as $n\to \infty.$
		This proves that   for all $2\gamma_n<b<+\infty$,
		$$\sup_{s\in [2\gamma_n,b]}  |G_n(s)-G(s)|\to 0\quad \text{ as }n\to \infty.$$
		Let us first assume  that  $ {\mathbf R}>0$, thus $G(t)=\Vert V-P^t\Vert_t=0$ for al $0\leq t\leq   {\mathbf R}$. Then from \eqref{cota1}
		\begin{equation}\label{cota00}
			\sup_{s\in [2\gamma_n,   {\mathbf R}]} G_n(s)<\varepsilon_n.
		\end{equation}
		From \eqref{bound0gamman}  $G_n(t)<2V(0)(2\gamma_n)^{1/2}$ for all $t\in [0,2\gamma_n]$. 
		This together with \eqref{cota00} proves that  for $n$ large enough,
		\begin{equation}\label{cota2}
			\tilde{\mathbf{R}}=	\inf\{t:G_n(t)>\varepsilon_n\}\geq  {\mathbf R}.
		\end{equation}
		Observe that if   $ {\mathbf R}=+\infty$  \eqref{cota2} proves that  $\tilde{\mathbf{R}} = +\infty$. 
		Assume now that $0\leq  {\mathbf R}<\infty$,  then  $G(t)>0$ for all $t> {\mathbf R}$. Let us fix $\delta> {\mathbf R}$ and $n$ large enough to ensure $\varepsilon_n<G(\delta)$.   $G_n(\delta)\to G(\delta)$ as $n\to \infty$, then   for all $n$ large enough,
		$$\tilde{\mathbf{R}}=\inf\{t:G_n(t)>\varepsilon_n\}\leq  \delta.$$
		Since this holds for all $\delta> {\mathbf R}$ it follows  that  
		\begin{equation}\label{cota4}
	 \lim_{n\to \infty} \tilde{\mathbf{R}}\leq {\mathbf R},\ \mbox { a.s.}
		\end{equation}
		This in particular proves that if $ {\mathbf R}=0$, $\tilde{\mathbf{R}}\to 0$  as $n\to \infty$.  From \eqref{cota2} and \eqref{cota4} it follows \eqref{limth}.
	\end{proof}
\end{theorem}

\subsection{An algorithm for a lower bound of  the polynomial reach}\label{subsec:algorithm}

Let us note that a consistent (in the statistical, asymptotic sense) estimator of ${\mathbf R}$, as that proposed in the previous subsection, might perhaps provide, for finite sample sizes, an over estimation of the polynomial reach ${\mathbf R}$.
Indeed, the convergence $\tilde {\mathbf R}\to {\mathbf R}$ does not exclude at all the possibility that $\tilde {\mathbf R}> {\mathbf R}$ for infinitely many values of $n$. 
 This might be particularly harmful in practice since, in  the over estimation case, the polynomial volume condition for $V$ is not fulfilled on $[0,\widehat {\mathbf R}]$ which might lead to poor estimation of the polynomial coefficients. From this point of view it is safer to look for a infra-estimation of ${\mathbf R}$. This is the purpose of the estimation algorithm we propose in  this subsection.

This algorithm (whose convergence will be proved below) works as follows

\begin{itemize}
	\item[] \textit{Inputs and notation.} A finite set $\aleph_n=\{X_1,\ldots,X_n\}\subset S$. An arbitrarily  small, $\eta>0$.  A grid of $K$ values $0<r_1<r_2<\dots<r_K$.  A sequence $U_n=(\log n/n)^{1/2d-\eta}$.  A positive integer value $\ell\geq d$ to be used as the degree of the approximating polynomials. 
	
	Let us define $I_i=[0,r_{i}]$   and $J_i=[r_i,r_K]$ for all $i=1,\dots,K-1$, let $P_{n,d}^{I_i}$ (resp. $P_{n,\ell}^{J_i})$ be the best polynomial approximation of $V_n$ of degree $d$ on the interval $I_i$ (resp. of degree $\ell\geq d$ on $J_i$).

	\item[] \textit{Step 0}.  Put $i=0$. If $\Vert V_n-P_{n,d}^{I_1}\Vert_{L^2(I_1)}>U_n$ then the output is $\widehat{\mathbf R}=0$ and the algorithm stops. Otherwise, put $i\leftarrow i+1$ and  go to the following step. 
	
	\textit{Step 1}  Define for  $i=1,\ldots,K-1$
	
	$$c_i=\frac{\Vert V_n-P_{n,d}^{I_{i}}\Vert_{L^2(I_{i})}}{\Vert V_n-P_{n,\ell}^{J_i}\Vert_{L^2(J_i)}}.$$
	
	\textit{Output}.
If the algorithm does not stop at Step 0, let $i$ be the first index such that $c_{i}>1$ The output of the algorithm is $\widehat {\mathbf R}=r_{i-1}$.  If there is no such $i$, we just define the output of the algorithm as $\widehat {\mathbf R}=r_{K-1}$.
	
\end{itemize}


This algorithm is studied in the following result which relies on two assumptions, denoted $H1$ and $H2$. In Theorem \ref{th.las.constantes} below we will show precise conditions under which $H1$ and $H2$ hold when $\aleph_n$ comes from a random sample whose support is $S$.
\begin{theorem}\label{th:consist}
	Let ${\mathbf R}$ be the polynomial reach of the compact set $S\subset {\mathbb R}^d$. Let us consider the above algorithm,  based on a grid $0=r_0<r_1<r_2<\dots<r_K$ when applied to any sequence of sets $\aleph_n=\{X_1,\dots,X_n\}\subset S$  such that $d_H(S,\aleph_n)\to 0$.
Assume 
	\begin{itemize}
		\item[\mbox{H1}.] For some $\eta>0$ we have  $\Vert V_n-P_{n}^{I_i}\Vert_{L^2(I_i)}<U_n:= (\log n/n)^{\frac{1}{2d}-\eta}$ for all $i>0$ such that $I_i\subset[0,\mathbf{R}]$,  for all $n$ large enough. 
		\item[\mbox{H2}.] For all $\varepsilon>0$, there exists a degree $\ell=\ell(\varepsilon)$, such that $\Vert V_n-P_{n,\ell}^{J_i}\Vert_{L^2(J_i)}<\varepsilon$, for all $i=1,\dots,K-1$, for all $n$ large enough. 
		\end{itemize}

\noindent Then, there exists $\ell$ such that  the output $\widehat{\mathbf R}$ of the algorithm fulfils
 	
 	\begin{itemize}
  \item[(i)] If $0\leq {\mathbf R}<r_1$ then $\widehat{\mathbf R}=0$ for all $n$ large enough.		If  ${\mathbf R}\geq r_K$ then $ r_1\leq \widehat {\mathbf R}\leq r_{K-1}$ for all $n$ large enough.
 	 \item[(ii)] 
 		 
 			If $r_{i_0}<\mathbf{R}<r_{i_0+1}$ for some $0<i_0<K$ then $\widehat{\mathbf R}=r_{i_0}$ for all $n$ large enough.
 			\item[(iii)]  If $\mathbf{R}=r_{i_0}$ for some $1\leq i_0 < K$ then  $\widehat {\mathbf R}\leq r_{i_0}$, moreover
 			if $\mathbf{R}=r_{i_0}>r_1$, then  $\widehat {\mathbf R}\in \{r_{i_0-1},r_{i_0}\}$, for all $n$  large enough.
 		\end{itemize}

	\begin{proof}
 (i)  Assume first that $0\leq {\mathbf R}<r_1$, let us prove that the algorithm stops in step 0. From the definition of ${\mathbf R}$,  $V$ is not a polynomial on the closed interval $I_1$. Then,   
		\begin{equation}
			\lim_n\Vert P_{n,d}^{I_1}-V_n\Vert_{L^2(I_1)}=\Vert P_d^{I_1}-V\Vert_{L^2(I_1)}>0\label{paso_0}.
		\end{equation}
		Indeed, $\Vert P_{n,d}^{I_1}-P_d^{I_1}\Vert_{L^2(I_1)}\rightarrow 0$, as a consequence of the continuity of the projections in a Hilbert space.  On the other hand, we must have $\Vert P_d^{I_1}-V\Vert_{L^2(I_1)}>0$. To prove this observe that  $V$ is polynomial of degree $d$ on $[0,\mathbf{R}]$ and not on $[0,r_1]$. So we must have $\Vert P_d^{I_1}-V\Vert_{L^2(I_1)}>0$. 
	    Let $n$ large enough to guarantee that $U_n<\Vert P_d^{I_1}-V\Vert_{L^2(I_1)}/2$.
    	We conclude \eqref{paso_0} and, in particular, that,   for all $n$ large enough,
		$$
	    \Vert P_{n,d}^{I_1}-V_n\Vert_{L^2(I_1)}>U_n.
		$$
		 So, the algorithm will stop at step 0 for all $n$ large enough, when $0\leq{\mathbf R}<r_1$ and then $0=\widehat{\mathbf R}\leq {\mathbf R}$ for all $n$ large enough as desired. 
		 
Now, if ${\mathbf R}>r_K$  assumption H1 entails that,  for all $n$ large enough,  the algorithm does not stop at the initial step. So, necessarily $r_1\leq \hat {\mathbf R}\leq r_{K-1}$. 
		
Before proving  (ii), assume that $r_{i_0}\leq {\mathbf R}<r_{i_0+1}$ for some $i_0\in\{1,\ldots, K-2\}$ and let us first prove that there exists $\ell$ large enough such that  $c_{i_0+1}>1$   for all $n$ large enough. Indeed, reasoning as we did with $I_1$, we conclude that for all $n$ large enough,  $\Vert P_{n,d}^{I_{i_0+1}}-V_n\Vert_{L^2(I_{i_0+1})}>\Vert P_d^{I_{i_0+1}}-V\Vert_{L^2(I_{i_0+1})}/2>0$. Thus,

\begin{align*}
	c_{i_0+1}&=\frac{\Vert V_n-P_{n,d}^{I_{i_0+1}}\Vert_{L^2(I_{i_0+1})}}{\Vert V_n-P_{n,\ell}^{J_{i_0+1}}\Vert_{L^2(J_{i_0+1})}} \geq 
	\frac{\frac{1}{2}\Vert P_{d}^{I_{i_0+1}}-V\Vert_{L^2(I_{i_0+1})}}{\Vert V_n-P_{n,\ell}^{J_{i_0+1}}\Vert_{L^2(J_{i_0+1})} }.
\end{align*}
 
By assumption H2 we can take $\ell$  such that  for all $n$ large enough,
\begin{equation}\label{ele}
 \Vert V_n-P_{n,\ell}^{J_{i_0+1}}\Vert_{L^2(J_{i_0+1})}<\frac{1}{2}\Vert P_{d}^{I_{i_0+1}}-V\Vert_{L^2(I_{i_0+1})} 
\end{equation}
 so that $c_{i_0+1}>1$. This in particular proves that if $\mathbf{R}=r_{i_0}$ for some $0<i_0 <K-1$ then  $\widehat {\mathbf R}\leq r_{i_0}$.
 
 Let us now prove $(ii)$. So, assume that $r_{i_0}< {\mathbf R}<r_{i_0+1}$ for $i_0\in \{1,\dots,K-1\}$. Take $\ell$ fixed but large enough to guarantee that \eqref{ele} holds. Let us prove that $c_i<1$ for all $1\leq i\leq i_0$. First observe that
\begin{equation}\label{lim1}
	\lim_{n\to \infty}\Vert V_n-P_{n,\ell}^{[r_{i_0},r_{i_0+1}]}\Vert_{L^2({[r_{i_0},r_{i_0+1}]})}=\Vert V-P_{\ell}^{[r_{i_0},r_{i_0+1}]}\Vert_{L^2([r_{i_0},r_{i_0+1}])}\quad a.s.
	\end{equation}
Let us prove that $V$ is not a polynomial of degree at most $\ell$ in $[r_{i_0},r_{i_0+1}]$.
Assume by contradiction that  $V(t)=\sum_{i=0}^\ell a_it^i$ for some $a_0,\dots,a_\ell$ and $t\in [r_{i_0},r_{i_0+1}]$. Since  $r_{i_0}<{\mathbf R}<r_{i_0+1}$  $V$ is a polynomial of degree at most $d$ in $[r_{i_0}, \mathbf{R}]$ and then $a_{d+1}=\dots = a_\ell =0$. But then $V$ is a polynomial of degree at most $d$ in $[0,r_{i_0+1}]$ from where it follows that $\mathbf{R}\geq r_{i_0+1}$ which is a contradiction.
Since $V$ is not a polynomial of degree at most $\ell$ in $[r_{i_0},r_{i_0+1}]$
$$\Vert V-P_{\ell}^{[r_{i_0},r_{i_0+1}]}\Vert_{L^2([r_{i_0},r_{i_0+1}])}>0.$$
  From \eqref{lim1} it follows that  for all $n$ large enough, 

$$\Vert V_n-P_{n,\ell}^{[r_{i_0},r_{i_0+1}]}\Vert_{L^2({[r_{i_0},r_{i_0+1}]})}>\frac{1}{2}\Vert V-P_{\ell}^{[r_{i_0},r_{i_0+1}]}\Vert_{L^2([r_{i_0},r_{i_0+1}])}>0$$
but for all $i\leq i_0$,
\begin{equation}\label{cimenorque1}
\Vert V_n-P_{n,\ell}^{J_{i}}\Vert_{L^2(J_{i})} \geq \Vert V_n-P_{n,\ell}^{[r_{i_0},r_{i_0+1}]}\Vert_{L^2({[r_{i_0},r_{i_0+1}]})}>\frac{1}{2}\Vert V-P_{\ell}^{[r_{i_0},r_{i_0}+1]}\Vert_{L^2([r_{i_0},r_{i_0+1}])}>0.
\end{equation}
Since 
$${\Vert V_n-P_{n,d}^{I_{i_0}}\Vert_{L^2(I_{i_0})}}\to 0$$
it follows from \eqref{cimenorque1} that $c_i\to 0$ for all $1\leq i\leq i_0$ because the right-hand side of \eqref{cimenorque1} does not depend on $n$ and $\ell$ is fixed, in particular $c_i<1$ for all $i\leq i_0< K$.  This in particular proves that if $r_{K-1}<\mathbf{R}<r_{K}$ then $\hat{\mathbf{R}}=r_{K-1}$. To conclude the proof of (ii) let us consider the case $r_{i_0}< {\mathbf R}<r_{i_0+1}$ for $i_0<K-1$; we have proved above that, in this situation we have $c_{i_0+1}>1$ for $i_{0}<K-1$, eventually. 
 Note that, in this case, the value $\widehat {\mathbf R}$ is eventually constant, $\widehat {\mathbf R}=r_{i_0}$. 

 Finally, to prove (iii) we must see that 	if $\mathbf{R}=r_{i_0}>r_1$, then  $\widehat {\mathbf R}\in \{r_{i_0-1},r_{i_0}\}$, eventually  almost surely.  As we have already proved  
 $\widehat {\mathbf R}\leq r_{i_0}$ for all $n$ large enough, it suffices to prove that $c_{i_0-1}<1$ (since, reasoning as in \eqref{cimenorque1}, we conclude $c_i<1$ for all $i<i_0-1$). First observe that, for all $\ell>0$ fixed,  for all $n$ large enough
 $$\Vert V_n-P_{n,\ell}^{J_{i_0-1}}\Vert_{L^2(J_{i_0+1})}>\frac{1}{2}\Vert V-P_{\ell}^{J_{i_0-1}}\Vert_{L^2(J_{i_0+1})} >0.$$
Also, $\Vert V_n-P_{n,d}^{I_{i_0-1}}\Vert_{L^2(I_{i_0-1})}\to 0$  for all  $n$ large enough. Then $c_{i_0-1}\to 0$.  So,  for all $n$ large enough $\hat {\mathbf R}\geq r_{i_0-1}$.

	\end{proof}
\end{theorem}

\subsection{On the assumptions H1 and  H2 of Theorem \ref{th:consist}}

We now address an obvious question: under which conditions on the set $S$ can we guarantee that assumptions H1 and H2 in Theorem \ref{th:consist} are fulfilled, when $\aleph_n$ is an iid sample? The answer is given in the next result.

\begin{theorem}\label{th.las.constantes}
	Under the assumptions of Lemmas \ref{lemaux2} and \ref{rates1} the algorithm defined in Subsection  \ref{subsec:algorithm} with $U_n= (\log n/n)^{\frac{1}{2d}-\eta} \mbox{ and } \eta>0$, provides, eventually with probability one, a lower bound for  ${\mathbf R}$.
\end{theorem}

\begin{proof} 
Let us prove that 	this choice of $U_n$ fulfils condition H1 of Theorem \ref{th:consist}. We have to show that if $\mathbf{R}\geq b$ then  for a fixed $\eta>0$,  $\Vert V_n-P_{n}^{I}\Vert_{L^2(I)}<U_n$ where $I= [0,b]$ with $\mathbf{R}\geq b$. Indeed,
	$$  \Vert V_n-P_{n,d}^{I}\Vert_{L^2(I)}\leq \Vert V-V_n\Vert_{L^2(I)}.$$
	So we will prove that  if $\mathbf{R}>b$ then  for any  fixed $\eta>0$,  $\Vert V-V_n\Vert_{L^2(I)}<U_n$. 
	This follows from Theorem 3 in  \cite{cue04}, which proves that with probability one, for all $n$ large enough, 
	$$d_H(\aleph_n, S)\leq \Big(\frac{2}{\delta\omega_d}\frac{\log n}{n}\Big)^{\frac{1}{d}},$$ together with Lemmas \ref{lemaux2} and \ref{rates1}.
	
	Let us prove that  hypothesis H2 of Theorem \ref{th:consist} holds. By Theorem 1 in  \cite{st76}   $V_n$ is absolutely continuous, then by  Lemma 3 in \cite{goli:91}, with probability one, for all $n$, 
	
	$$\Vert V_n-P_{n,\ell}^{J_i}\Vert_{L^2(J_i)}\leq \frac{\sqrt{2}\pi}{2(\ell+1)}\Vert V_n'\Vert_{L^2(J_i)}$$
	for all $i=1,\dots,K-1$, where $\ell$ is even.
	To bound $\Vert V_n'\Vert_{L^2(J_i)}$ we will use again \citet[p. 1665]{rataj10}: in the points $t$ where $V_n'(t)$ exists, it holds that 
	$V_n'(t)= {\mathcal H}^{d-1}(\partial B(\aleph_n,t))\leq d V_n(t)/t$ for $t>0$.
	Since for all $t>0$ $0\leq V_n(t)\leq V(t)$ it follows that 
	$$
	\Vert V_n-P_{n,\ell}^{J_i}\Vert_{L^2(J_i)}\leq \frac{\sqrt{2}\pi}{2(\ell+1)}dV(r_K)\sqrt{\int_{r_i}^{r_K}  1/t^2dt}\leq \frac{\sqrt{2}\pi}{2(\ell+1)} dV(r_K) \sqrt{\frac{1}{r_1}-\frac{1}{r_K}}
	$$
and then H2 holds.

\end{proof}

\section{Estimation of the polynomial coefficients: convergence rates}

Theorem 1 in \cite{cuepat18} states that the coefficients of the best polynomial approximation of the estimated volume function $V_n$,  are consistent estimators of the coefficients of $V$.  In this section we obtain the rates of convergence for those estimators. First we will assume that either $\mathbf{R}$ is known, or we have an underestimation. 
 
In order to do that, we will use the following result.

\begin{lemma}\label{boundcoef} Let $[a,b]\subset \mathbb{R}$. There exists a constant $\kappa_d>0$ such that for any pair of polynomials $f(t)=\sum_{i=0}^d \alpha_i t^i$ and $g(x)=\sum_{i=0}^d \beta_i t^i$ defined on $[a,b]$,
	$$| \alpha_i - \beta_i  |\leq \kappa_d \Vert f-g\Vert_{L^2([a,b])}.$$
	Moreover, $\kappa_d$ depends only on $d$ and $[a,b]$.
	\begin{proof}
This is an immediate consequence of the fact that $\Vert \alpha\Vert_\infty$ and $\Vert\sum_{i=0}^d \alpha_ix^i\Vert_{L_2([a,b])}$ are two norms on $\mathbb{R}^{d+1}$. Hence they are equivalent. 
	\end{proof}
\end{lemma}

The following two results provide the convergence rates for the estimation of the polynomial coefficients. Not surprisingly, these rates are of ``non-parametric type" with orders $O((\log n/n)^{1/d})$ depending on the dimension.

\begin{proposition} \label{lemhausdist} Let $S\subset \mathbb{R}^d$ be compact such that the Hausdorff dimension of $S$ is and integer value $d'$ with $0<d'\leq d$. Let $X$ be a random variable with support $S$. Assume that $S$ fulfils the standardness condition \eqref{estandar} for $\nu=P_X$ and $\mu_d$ replaced with $\mathcal{H}^{d'},$ the $d'$-dimensional Hausdorff measure on $S$.  Assume further that $S$ has polynomial reach $\mathbf{R}>0$. Let $\aleph_n$ be an iid sample of $X$. Let $P_{n,d}^{I}(t)$ be as in \eqref{eq2}, for $t\in I\subset (0,\mathbf{R}]$.  Then, there exists $\kappa_d>0$ such that, with probability one,
	\begin{equation}\label{eqlemhausdist}
		\limsup_{n\to \infty} \Bigg(\frac{n}{\log n }\Bigg)^{1/d'}\max_{i} |\theta_i-\theta_{in}|\leq 2\kappa_d C\Bigg(\frac{2}{\delta\omega_d}\Bigg)^{1/d'}
	\end{equation}
$\delta$ being the standardness constant given by \eqref{estandar}, and $C>0$. 
\end{proposition}
\begin{proof}
	Given $s>0$ we have that, for $n$ large enough, 
	$$B(S,s-d_H(\aleph_n,S))\subset B(\aleph_n,s)\subset B(S,s).$$
	This follows from the fact that for any $z\in B(S,s-d_H(\aleph_n,S))$ and $x\in S$, if we denote $z^*$ the closest point to $z$ in $S$,
	$
	d(z,\aleph_n)\leq d(z,z^*)+d(z^*,\aleph_n)\leq s-d_H(\aleph_n,S)+d_H(\aleph_n,S)=s.
	$
	Then,
	$$V(s)-V_n(s)\leq V(s)-V(s-d_H(\aleph_n,S)).$$
	Now, using that $V$ is Lipschitz on $I$ (see the proof of Lemma \ref{lemaux2} for details)
	$\Vert V-V_n\Vert_{L^2(I)}\leq C d_H(\aleph_n,S)$, for some $C>0$. Now,  it can be proved, following the same ideas used in the proof of Theorem 3 in \cite{cue04} that,  with probability one,
	$$\lim_{n\to \infty} \Bigg(\frac{n}{\log n }\Bigg)^{1/d'}d_H(\aleph_n,S)\leq \Bigg(\frac{2}{\delta\omega_d}\Bigg)^{1/d'}.$$
From Lemma \ref{boundcoef},
	\begin{align*}
		&	\max_{i} |\theta_i-\theta_{in}|\leq \kappa_d\Vert V-P_{n,d}^{I}(t)\Vert_{L^2(I)}\\ 
		& \leq \kappa_d \Vert V-V_n\Vert_{L^2(I)}+\kappa_d\Vert V_n-P_{n,d}^{I}(t)\Vert_{L^2(I)}
		\leq 2 \kappa_d\Vert V-V_n\Vert_{L^2(I)}
	\end{align*}
	 which proves \eqref{eqlemhausdist}.
\end{proof}

Moreover, since the output of the algorithm given in subsection \ref{subsec:algorithm} is constant for $n$ large enough, we can obtain the rates  of convergence for the estimators of the coefficients as given by the next result.

 \begin{theorem}\label{th:coef} Under the assumptions of Proposition \ref{lemhausdist}, let $\widehat{\mathbf{R}}$ be the output of the algorithm given in  subsection \ref{subsec:algorithm}.  Assume $r_{i_0}<{\mathbf R}<r_{i_0+1}$ for some  $i_0\in\{1,\ldots.K-1\}$.  Let $P_{n,d}^{[r_1,\widehat{\mathbf{R}}]}(t)=\sum_{i=0}^d \theta_{in}t^i$ for $t\in [r_1,\widehat{\mathbf{R}}]$ as in \eqref{eq2}. Then, with probability one, 
	  $$\max_{i} |\theta_i-\theta_{in}|=O\Big(\Big(\frac{\log n }{n}\Big)^{1/ d'}\Big).$$
\end{theorem}
\begin{proof} In the proof of Theorem \ref{th:consist} it is shown that in the case $r_{i_0}<{\mathbf R}<r_{i_0+1}$ the output $\hat {\mathbf R}$ of the algorithm is eventually constant, almost surely and a lower bound of the true value ${\mathbf R}$. Then this result is a direct consequence of Lemma \ref{boundcoef} and the fact that  $\Vert V-V_n\Vert_{L^2(I)}\leq C d_H(\aleph_n,S)$. 
\end{proof}

\begin{remark} When $S$ is full-dimensional  and we only have an iid sample in the set, the estimation rate for \(\theta_1\) -the Minkowski content of the boundary of \(S\)- achieved in Theorem \ref{th:coef} matches the one derived in \cite{aar22} using Crofton's formula, where it is assumed that \(\partial S\) is \(C^2\) (hence it has positive Federer's reach), together with the assumption that there is an upper bound for the number of points in the intersection of any straight line and $\partial S$. As shown in \cite{aar22} this rate can be further improved to \(C(\log n /n)^{2/(d+1)}\) for sets with  \(C^2\) boundary, using as a plug-in estimator of $\theta_1$ the \(d-1\)-dimensional Hausdorff measure of the boundary of the \(r\)-convex hull of the sample. 
	However, to our knowledge, there is currently no practical algorithm for computing this estimator when $d>2$.
	
\end{remark}

\section{Computational aspects. Numerical experiments}\label{compasp}

In this section we will study the performance of the algorithm presented in subsection \ref{subsec:algorithm} for three examples of polynomial volume sets in the plane: 
\begin{itemize}
	\item[1]   The ``Pacman set'', defined as $S_P=B(0,1)\cap H^c$  where $H=\{(x,y): x>0 \text{  and }y>0\}$. This set is shown in Figure \ref{pacman} together with   two parallel sets, $B(\aleph_n,0.08)$  and $B(\aleph_n,0.3)$, of a sample $\aleph_n$ drawn on $S$. It can be shown that in this case the polynomial reach is $\mathbf{R}=1$ and its volume function, for $0\leq t\leq 1$, is the polynomial 
	$$V(t)=(5\pi/4-1)t^2+2(1+3\pi/4)t+3\pi/4.$$
\item[2]  The ``Union-of-squares'':  $S_U=[2,4]\times [-1,1]\cup [-1,1]^2$. Its polynomial reach is $\mathbf{R}=1$ and its volume function, for $0\leq t\leq 1$, is the polynomial $V(t)=2\pi t^2+16\pi t+8$. In Figure \ref{cuad} it is shown  the set, and two parallel sets of the samples $B(\aleph_n,0.1)$   and $B(\aleph_n,0.7)$.
\item[3]  The  ``Frame'' set $S_F=[-1,1]^2\setminus M$ where $M\subset (-1,1)^2$ is a square with side length $F<1$. Its polynomial reach is $\mathbf{R}=F/2$ and its volume function is
$$V(t)=\begin{cases}
	4-F^2+4(F+2)t+(\pi-4)r^2& t\in [0,F/2]\\
	4+8r+\pi r^2&t>F/2
\end{cases}.$$
We took $F=1$ for the simulations. 
\end{itemize}

In the  previous section we have studied  the problem of the estimation of the coefficients when the polynomial is adjusted in the interval $[0,\widehat {\mathbf R}]$, $\widehat {\mathbf R}$ being the output of our algorithm.  

It is easy to demonstrate that the volume function for  $S_P$ is differentiable at $t=1$, while the volume function for $S_U$ is not differentiable at $t=1$.  This makes $S_P$ a more challenging example for estimating its polynomial reach, as we have to detect a smoother change point. 

 \begin{figure}
 	\centering
 	\includegraphics[width=0.4\linewidth]{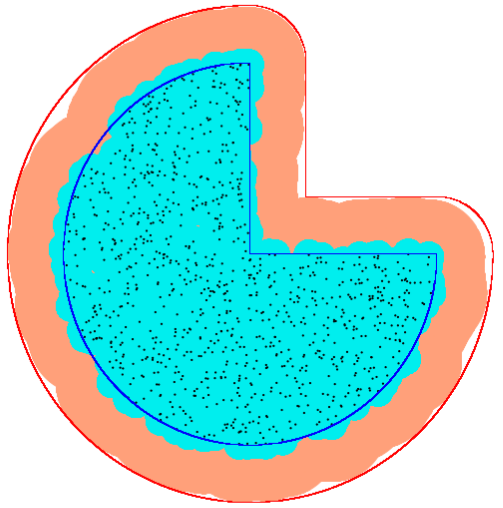}
 	\caption{ In blue solid line, the boundary of the ``Pacman set''; in red solid line the boundary of the corresponding dilation (parallel set) with radius $0.3$.   The dilations, with radii $0.08$ and $0.3$, of a sample of size $n=1000$ correspond to the light blue region and the light blue+orange  region, respectively.}
 	\label{pacman}
 \end{figure}
 
  \begin{figure}
 	\centering
 
 	\includegraphics[width=0.8\linewidth]{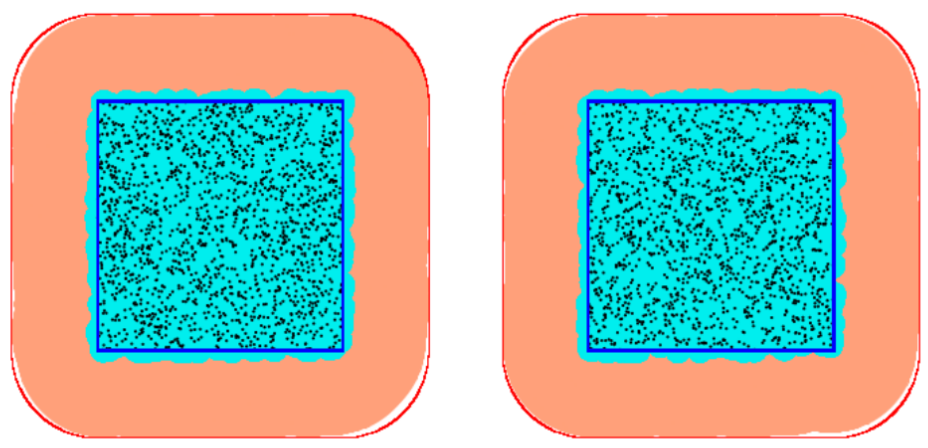}
 	\caption{ In red solid line the boundary of the parallel set of radius $0.7$. In light blue it is shown the dilatation of a sample of size $n=3000$, of radius $0.1$, and the region in orange plus the region in light blue is the dilatation of the sample of radius 0.7 for the union of squares.}
 	\label{cuad}
 \end{figure}

 For both $S_P$ and $S_U$ we have applied the algorithm using three  different grids of values $0<r_1^g<\dots<r_{K_g}^g$ for $g=1,2,3$, denoted by gr1, gr2 and gr3. In the three cases $r_K=1.98$. For $g=1$ and $g=3$, we took $r^g_j-r^g_{j-1}=0.4$  for all $j=2,\dots,r_{K_2}^2$, $r_1^1=0.2$ and $r_1^3=0.3$. For gr2  $r^2_j-r^2_{j-1}=0.3$  for all $j=2,\dots,r_{K_g}^g$ and $r^2_1=0.3$. The grid choice is a crucial point in the algorithm. Our numerical experiments suggest to avoid small values of the increments $r^g_j-r^g_{j-1}$, which could lead to numerical instabilities associated with the estimation of the $L^2$ norms.

In the case of $S_{F}$, where $\mathbf{R}$ is obviously smaller than in the other examples (in fact it is ${\mathbf R}=0.5$) we have used other different grids (denoted again g1, gr2 and gr3) with values $0<r_1^g<\dots<r_{K_g}^g$ for $g=1,2,3$. In this case gr1 is a grid from $r_1^1=0.1$ to  $r_{K_1}^1=1.5$ and the grid step is  $0.2$, gr2 is a grid from $r^2_1=0.1$ to $r_{K_2}^2=1.5$ and the grid step is $0.25$ and gr3 goes from $r^3_1=0.2$ to $r_{K_3}^3=1.5$ with a grid step of  $0.3$. 
 
The computations of $P_{n,d}^{I_i}$ and $P_{n,\ell}^{J_i}$ are performed by using Bernstein polynomials. The  norms are evaluated over grids with step $0.001$.

 As a consequence of the use of three grids, for each simulation run, we have obtained three values of $\widehat{\mathbf{R}}$. In all cases, the calculation of $V_n$ has been done by a Monte-Carlo approximation using a sample of one million points on a square containing the sets.  
 
The outputs for $\widehat{\mathbf{R}}$ in $S_P$, $S_U$ and $S_F$ are summarized in Tables \ref{gr8_pacman}, \ref{cuadrados} and \ref{marco}, respectively (see Appendix B).

  These results suggest that, in all instances, the increase in the parameter $\ell$, reduces the number of cases where $\mathbf{R}$ is overestimated though, as we will see below, this does not always improves the coefficient estimation. On the other hand, the number of overestimations depends on the grid, ranging from 0 to 49, even with 4000 data points (see Table \ref{cuadrados}).
  As expected, the estimation of the reach for ``frame set'' $S_F$ is a more challenging problem than for the $S_P$ or $S_U$, since the reach is smaller, which requires use of finer grids and larger sample sizes. This is also seen in the coefficients estimation.

 \subsection{The case where $\mathbf{R}$ is small}
 
To demonstrate the algorithm's performance when $\mathbf{R}$ is small, we have considered the ``Union-of-squares'' set $S_U$,  where both squares are separated by a distance of $2\lambda$ (thus, $\mathbf{R}=\lambda$). When $r_1>2\lambda$ (with $\lambda=0.01, 0.05$), the algorithm should halt at step 0. For $U_n$, we set $\eta=1/10$. As shown in Table \ref{R_small}, the algorithm generally stops at case 0. Intuitively, this happens because the empirical volume $V_n$ significantly deviates from a polynomial of degree 2 even over  small intervals starting at 0.

 \subsection{Estimation of the coefficients}

   To assess the coefficient estimation,  we run the algorithm 1000 times. The coefficients are estimated by least squares  over a grid  from $0.1$ to $\widehat{\mathbf{R}}$ with a grid step 0.01.  The reason to start at 0.1 is to avoid the obvious infra-estimation of $V(s)$ provided by $V_n(s)$ when $s$ is small (recall that $V_n(0)=0$). Also, we have excluded in the least squares estimation process those runs where $\widehat {\mathbf R}$ coincides with the minimum possible value $r_1$ (see the algorithm above), in order to avoid use of too small intervals leading to inaccurate estimations. The outputs are given  in Tables \ref{coef_pacman5000} and \ref{coef_pacman7000} for $S_P$, Tables \ref{coef_sq5000} and \ref{coef_sq7000} for $S_U$, and in Tables \ref{marco5000} and \ref{marco7000} for $S_F$.

   The results show that the estimation of the first two polynomial coefficients, which are our main target here, are more accurate than the estimation of the highest degree coefficient. On the other hand, as expected, the estimations are better in those grids where the algorithm estimates ${\mathbf R}$ with a higher average value. Lastly, as with the estimation of R, the case of the $S_F$ set is harder and requires more data.

\section{Some conclusions} 

We place ourselves in the, quite broad, context of sets $S$ having a polynomial volume function on some interval $[0,r]$
We deal with the estimation of the polynomial reach, that is, the supremum ${\mathbf R}$ of the values $r$ for which the polynomial condition holds. The available information is just a random sample inside the set $S$ of interest. 

Some  plausible consistent estimators, as that considered in Proposition \ref{prop:consist} might have a poor practical performance. Thus, since the interval $[0,{\mathbf R}]$ of polynomial reach is mainly used in practice to estimate the relevant coefficients of the polynomial volume, a conservative infra-estimation, as that analyzed in Theorem \ref{th:consist}, might be enough. 

The minimum-distance estimators of the polynomial coefficients yield, as a by-product,  estimators of the volume and the surface measure of the set $S$. It is important to note that, unlike other volume and surface area estimators available in the literature (see some references below), this procedure requires only an inside sample, rather than two samples, inside and outside the set. Also, it holds under very general conditions for the set $S$: no convexity or $r$-convexity is assumed, no positive reach or rolling condition is imposed. Finally, the estimation method does not essentially relies on a smoothing parameter (though some tuning parameters are unavoidably involved in the algorithm). For some references and background on volume/surface area estimation see \cite{ari19}, \cite{bal21}, \cite{cue07}, \cite{cue09}, \cite{cue12}, \cite{cue12} and \cite{jim11}.

Quite predictably, given the nature of the problems at hand, the proposed methods only work in practice with large samples sizes (around 5000 sample points) which is not a serious limitation in cases where we adopt a Monte Carlo approach, based on simulated samples, to estimate the geometric parameters (polynomial reach, volume, surface area) and the set is in fact known. 

As an interesting open problem we could mention the development of statistical tests of low dimensionality, based on the estimators of the polynomial coefficients.

\appendix

\section{Proofs of Lemmas 1 and 2}
 
\begin{proof}[Proof of Lemma \ref{lemaux2}]
	Observe that for all $n$ large enough,  $\max_{p\in \partial S}d(p,\aleph_n)<\gamma_n<b/2$.  
	Let us prove that, for all $s\in [2\gamma_n,b]$, 
	\begin{equation}\label{lemaux2eq}
		B\Big(S,s-\max_{p\in \partial S}d(p,\aleph_n)\Big)\subset B(\aleph_n,s)\subset B(S,s).
	\end{equation}
	We have to prove only the first inclusion, the second one follows from $\aleph_n\subset S$.
	Let $z\in B(S,s-\max_{p\in \partial S}d(p,\aleph_n))$. Since  $s\geq 2d_H(S,\aleph_n)$, $S\subset B(\aleph_n,s)$. So if  $z\in S$ we have $z\in B(\aleph_n,s)$. 
	Let us consider the case $z\notin S$.  Let $z^* \in \partial S$ the projection onto $\partial S$ of $z$. Since $z\notin S$, $d(z,z^*)\leq s-\max_{p\in \partial S}d(p,\aleph_n)$.
	Since $z^*\in \partial S$ there exists $X_i\in \aleph_n$ such that $d(z^*,X_i)\leq \max_{p\in \partial S}d(p,\aleph_n)$.
	Then
	$$d(z,\aleph_n)\leq d(z,z^*)+d(z^*,\aleph_n)\leq  s-\max_{p\in \partial S}d(p,\aleph_n)+ \max_{p\in \partial S}d(p,\aleph_n)\leq s,$$
	and we conclude that $z\in B(\aleph_n,s)$. This proves \eqref{lemaux2eq}  and, for $s\in[2\gamma_n,b]$,
	\begin{equation}\label{dif-vols}
		V(s)-V_n(s)\leq \mu\Bigg[B(S,s)\setminus  B\Big(S,s-\max_{p\in \partial S}d(p,\aleph_n)\Big)\Bigg]=V(s)-V\Big(s-\max_{p\in \partial S}d(p,\aleph_n)\Big).\end{equation}
	
	Let us now prove that $V$ is Lipschitz.   First, recall that, as $V$ is monotone, the derivative $V'(t)$ exists except for a countable set of points $t$. Also, when $V'(t)$ does exist, it coincides with, ${\mathcal H}^{d-1}(\partial B(S,t))$, the $d-1$-dimensional Hausdorff measure of the boundary of the parallel set $B(S,t)$, see \citet[Cor. 2.5]{rataj10}.  But, as shown in \citet[p. 1665]{rataj10}, ${\mathcal H}^{d-1}(\partial B(S,t))\leq d (V(t)-V(0))/t$ for $t>0$. 
	
	Now, as we are assuming that $V^{+}(0)$ is finite, the upper limit of ${\mathcal H}^{d-1}(\partial B(S,t))$, as $t\to 0$ is also finite, so that ${\mathcal H}^{d-1}(\partial B(S,t))$ must be finite almost everywhere in a neighborhood $[0,\chi]$ for some $\chi>0$.  Therefore, there is a constant $L_1$ such that $|V'(t)|<L_1$ on $[0,\chi]$.
	On the other hand, according to Theorem 1 in \citet{st76}, for all $\varepsilon>0$, the volume function $V$ is absolutely continuous on $[\varepsilon, b]$ and its derivative can be expressed, almost everywhere, as $t^{d-1}\alpha(t)$, for some monotone decreasing function $\alpha$. As a consequence there is a positive constant $C$ such that $|V'(t)|<C$ almost everywhere on $[0,b]$.  We thus conclude that $V$ is Lipschitz continuous on $[0,b]$ with Lipschitz constant $C$.
	
Now, if we restrict ourselves to the interval $[0,2\gamma_n]$, 
	\begin{equation*}\label{bound2gamman}
		\Vert V-V_n\Vert_{L^2([0,2\gamma_n])}\leq (V(2\gamma_n)^22\gamma_n)^{1/2}	= \Big( V(0)+V^+(0)2\gamma_n+o(2\gamma_n)\Big)(2\gamma_n)^{1/2} \leq  2V(0)(2\gamma_n)^{1/2},
	\end{equation*}
	which proves \eqref{bound0gamman}. This, together  with and \eqref{dif-vols} and the Lipschitz  property of $V$, yields 
	\begin{align*}
		\Vert V-V_n\Vert_{L^2([0,b])}\leq & \Vert V-V_n\Vert_{L^2([0,2\gamma_n])}+\Vert V-V_n\Vert_{L^2([2\gamma_n,b])}\\
		\leq & 2V(0)(2\gamma_n)^{1/2}+\sqrt{b}C\max_{p\in \partial S}d(p,\aleph_n)
	\end{align*}
for all $n$ large enough.
\end{proof}
 
\begin{proof}[Proof of Lemma \ref{rates1}] Let us prove that from  $L_0(\partial S)<\infty$ it follows that there  exist $\varepsilon_0$ small enough such that we can cover $\partial S$ by $6L_0(\partial S)/(\omega_{d}\eps^{d-1})$ balls of radius $3\eps$ centered at $\partial S$ for all $0<\eps<\eps_0$. To prove this let us cover $\partial S$ with a minimal covering of radius $3\eps$ centred at points $\{x_1,\dots,x_k\}\subset \partial S$. Then $B(x_i,\eps)\subset B(\partial S, \eps)$ and $B(x_i,\eps)\cap B(x_j,\eps)=\emptyset$ for all $i\neq j$. Then $\mu(B(\partial S,\eps))\geq k\omega_d \eps^d$. Thus, for $\eps$ small enough, $\mu(B(\partial S,\eps))/(2\eps)\leq 2L_0(\partial S)$ and $k\leq 6L_0(\partial S)/(\omega_d\eps^{d-1})$. 

Now note that, given $\varepsilon>0$ small enough, $\max_{p\in \partial S}d(p,\aleph_n)>\varepsilon$ entails that at least one of the balls $B(x_i,\varepsilon)$ does not contain any sample point, thus
\begin{equation}\label{eq:eps}
{\mathbb P}\Big(\max_{p\in \partial S}d(p,\aleph_n)>\varepsilon\Big)\leq k\Big(1-\delta\omega_d\varepsilon^d\Big)^n\leq \frac{6L_0(\partial S)}{\omega_d\eps^{d-1}}\Big(1-\delta\omega_d\varepsilon^d\Big)^n
\end{equation}
Let us denote   $\zeta_n=(n/\log n )^{1/d}$,  and $\kappa=(2/(\delta\omega_d))^{1/d}$. Using \eqref{eq:eps} for $\varepsilon=\kappa/\zeta_n$, we have
	$$\mathbb{P}\Big(\zeta_n\max_{p\in \partial S}d(p,\aleph_n)> \kappa \Big)\leq \zeta_n^{d-1}6\frac{L_0(\partial S)}{\omega_d\kappa^{d-1}}\Big(1- \frac{2}{\zeta_n^d}\Big)^n= \zeta_n^{d-1}6\frac{L_0(\partial S)}{\omega_d\kappa^{d-1}}\Big(1- \frac{\log n^2}{n}\Big)^n\leq \zeta_n^{d-1}6\frac{L_0(\partial S)}{\omega_d\kappa^{d-1}}\frac{1}{n^2}$$
	
	The result follows now from Borel-Cantelli's lemma.
\end{proof}

 \section{Tables}
 
 \begin{table}[hbt!]
 	\footnotesize
 	\centering
 	\begin{tabular}{cccc|ccc|ccc}
 		& \multicolumn{3}{c}{$n=2000$}   & \multicolumn{3}{c}{$n=3000$}  & \multicolumn{3}{c}{$n=4000$} \\
 		     &  gr1  &  gr2  &  gr3 & gr1  &  gr2  &  gr3& gr1  &  gr2  &  gr3 \\ \hline
 		Mean & 0.6212 & 0.7692 & 0.7196 &0.606  & 0.7431 &   0.7104 &0.6024 & 0.738  &   0.704  \\
 		S.d. & 0.2880 &  0.15  & 0.0877 &0.0519 & 0.1496 &   0.0632 & 0.03  & 0.1493 &   0.0387\\
 		     &   0    &   2    &   50 &0    &   0    &     26 & 0    &   0    &     10 \\ \hline
 		Mean & 0.6096 & 0.7398 & 0.7096 &0.6024 & 0.7149 & 0.7044  & 0.6016 & 0.6954 & 0.7024\\
 		S.d. & 0.0656 & 0.1507 & 0.2236 &0.04   & 0.1456 & 0.1516 & 0.0245 & 0.1396 & 0.03 \\
 		     &   0    &   0    &   24  &0    &   0    &   13&0    &   0    &   6
 	\end{tabular}
 	\caption{\footnotesize Mean and standard deviation of $\widehat{\mathbf{R}}$ for $S_P$ over 1000 replications, for each of the three grids considered.  We took $\ell=8$ for rows 1,2,3 and $\ell=10$ for rows  4,5,6.   In rows 3 and 6  we show the number of times in the 1000 replications where the algorithm provides an overestimation of the true value $\mathbf{R}=1$.}
 	\label{gr8_pacman}
 \end{table}

 \begin{table}[hbt!]
 	\footnotesize
 	\centering
 	\begin{tabular}{cccc|ccc|ccc}
 		& \multicolumn{3}{c}{$n=2000$}  & \multicolumn{3}{c}{$n=3000$} &\multicolumn{3}{c}{$n=4000$} \\
 &  gr1  &  gr2  &  gr3 & gr1   &  gr2  &   gr3  & gr1   &  gr2  &   gr3     \\ \hline
 		Mean & 0.5984 & 0.696  & 0.7116 & 0.5948  & 0.6735 &  0.6956 & 0.602  & 0.7026 &   0.7084   \\
 		S.d.  & 0.0819 & 0.1432 & 0.1044 &0.0630  & 0.1311 &  0.0917 &0.0866 & 0.1435 &  0.1077   \\
 		&   0    &   3    &   49&0    &   2    &    21 & 0    &   1    &     47 \\  \hline
 		Mean & 0.5956 & 0.6792 & 0.698  &0.5856 & 0.6624 & 0.6884&0.5828 & 0.648  & 0.6828 \\
 		S.d.  & 0.0843 & 0.1357 & 0.1049 &0.0917 & 0.1237 & 0.0877 &0.0849 & 0.1140 & 0.0954 \\
 		&   0    &   2    &   32&0    &   0    &   10&0    &   0    &   8
 	\end{tabular}
 	\caption{\footnotesize Mean and standard deviation of $\widehat{\mathbf{R}}$ for $S_U$ over 1000 replications, for each of the three grids considered.  We took $\ell=8$ for rows 1,2,3 and $\ell=10$ for rows  4,5,6.   In rows 3 and 6  we show the number of times in the 1000 replications that the algorithm provides an overestimation of the true value $\mathbf{R}=1$.}
 	\label{cuadrados}
 \end{table}

  \begin{table}[hbt!]
 	\footnotesize
 	\centering
 	\begin{tabular}{cccc|ccc|ccc}
 		& \multicolumn{3}{c}{$n=5000$}  & \multicolumn{3}{c}{$n=7000$} &\multicolumn{3}{c}{$n=9000$} \\
 &  gr1  &  gr2  &  gr3 & gr1   &  gr2  &   gr3  & gr1   &  gr2  &   gr3     \\ \hline
 		Mean &  0.511  & 0.4265 & 0.494  & 0.5232 & 0.432  &   0.4973  & 0.5242 & 0.4415 &   0.5012   \\
 		S.d.  & 0.1073 & 0.1180 & 0.0625 &0.1025 & 0.1183 &   0.0663  &0.0985 & 0.1215 &   0.0465   \\
 		&   171   &  309   &    192   &  330   &     20  &   10&187   &  366   &     14\\  \hline
 		
 		Mean & 0.4698&  0.387 & 0.4709  &0.4766 & 0.3865  &  0.4739  &	0.4816 & 0.398  &0.4823 \\
 		S.d.  &   0.1128&  0.0982&0.0898 &0.1054 & 0.0926 &  0.0887 & 0.1004&  0.1041  &0.0732  \\
 		&    95  &  162   &   1   &  87   &   152   &     4 & 84 & 201 & 2
 	\end{tabular}
	\caption{\footnotesize Mean and standard deviation of $\widehat{\mathbf{R}}$ for $S_F$ over 1000 replications, for each of the three grids considered.  We took $\ell=30$ for rows 1,2,3 and $\ell=50$ for rows  4,5,6.   In rows 3 and 6  we show the number of times in the 1000 replications that the algorithm provides an overestimation of the true value $\mathbf{R}=0.5$.}
	\label{marco}
 \end{table}

  \begin{table}[hbt!]
 	\footnotesize
 	\centering
 	\begin{tabular}{cccc|ccc|ccc}
 		& \multicolumn{3}{c}{$n=1000$, $\lambda=0.1$}  & \multicolumn{3}{c}{$n=1500$, $\lambda=0.05$} &\multicolumn{3}{c}{$n=1800$, $\lambda=0.05$}  \\
$r_1=$ & 0.15 & 0.2 &         0.25    &0.12 & 0.15 &         0.20  & 0.12 & 0.15 &         0.20 \\ \hline
 &  64  &  0  &          0&118  &  0   &          0&  0   &  0   &          0
 	\end{tabular}
		\caption{\footnotesize Number of times, in 1000 replications, that the algorithm wrongly does not stop at step $0$ for the ``Union of squares'' with distance $2\lambda$ between both squares.}
		\label{R_small}
 \end{table}

\begin{table}[hbt!]
	\footnotesize
	\centering
	\begin{tabular}{|ccc|ccc|ccc|l}
		\multicolumn{3}{c|}{gr1} & \multicolumn{3}{c|}{gr2} & \multicolumn{3}{c|}{gr3} & true  values \\
		Mean  &   S.d.   &   MAD   &  Mean  &   S.d.   &   MAD   &  Mean  &   S.d.   &   MAD   &              \\
		2.3117 & 0.1989 & 0.1599  & 2.2916 & 0.1908 & 0.1520  & 2.3133 & 0.1946 & 0.1576  & 2.3562       \\
		6.7734 & 1.1055 & 0.8831  & 6.9120 & 0.9415 & 0.7355  & 6.7616 & 0.9610 & 0.7695  & 6.7124       \\
		2.8358 & 1.5069 & 1.2089  & 2.6411 & 1.1829 & 0.8903  & 2.8536 & 1.1547 & 0.9374  & 2.9270
	\end{tabular}
	\caption{\footnotesize Mean, standard deviation (S.d.) and median absolute deviation (MAD) of the estimated polynomial coefficients for $S_P$, over 1000 replications  with $n=5000$. In the last column we show the theoretical true values. The algorithm is applied with $\ell=8$.}
	\label{coef_pacman5000}
\end{table}

 \begin{table}[hbt!]
	\footnotesize
	\centering
	\begin{tabular}{|ccc|ccc|ccc|l}
		\multicolumn{3}{c|}{gr1} & \multicolumn{3}{c|}{gr2} & \multicolumn{3}{c|}{gr3} &true values  \\
		Mean  &  S.d.    &   MAD   &  Mean  &  S.d.   &   MAD   &  Mean  &  S.d.   & MAD     &               \\
		2.3280 & 0.1685 & 0.1348 & 2.3149 & 0.1665 & 0.1326 & 2.3334 & 0.1666 & 0.1338 & 2.3562 \\
		6.8039 & 0.9296 & 0.7344 & 6.8939 & 0.8400 & 0.6679 & 6.7627 & 0.8102 & 0.6469 & 6.7124 \\
		2.7740 & 1.2886 & 1.0202 & 2.6493 & 1.0897 & 0.8265 & 2.8385 & 0.9802 & 0.7753 & 2.9270 \\
		
	\end{tabular}
	\caption{\footnotesize Mean, standard deviation (S.d.) and median absolute deviation (MAD) of the estimated polynomial coefficients for $S_P$, over 1000 replications, with $n=7000$. In the last column we show the theoretical true values. The algorithm is applied with $\ell=8$.}
	\label{coef_pacman7000}
\end{table}

\begin{table}[hbt!]
	\footnotesize
	\centering
	\begin{tabular}{|ccc|ccc|ccc|l}
		\multicolumn{3}{c|}{gr1} & \multicolumn{3}{c|}{gr2} & \multicolumn{3}{c|}{gr3} &true values \\
		Mean   &  S.d.  &  MAD   &  Mean   &  S.d.  &  MAD   &  Mean   &  S.d.  &  MAD   &        \\
		7.7389  & 0.5237 & 0.4182 & 7.7240  & 0.5159 & 0.4140 & 7.7438  & 0.5268 & 0.4211 & 8      \\
		16.3034 & 2.6660 & 2.1703 & 16.3761 & 2.6315 & 2.1187 & 16.2319 & 2.4377 & 1.9510 & 16     \\
		5.7662  & 3.5644 & 2.8850 & 5.6585  & 3.3860 & 2.7323 & 5.8778  & 2.8724 & 2.3105 & 6.2832
	\end{tabular}
	\caption{\footnotesize Mean, standard deviation (S.d.) and median absolute deviation (MAD) of the estimated polynomial coefficients for $S_U$, over 1000 replications, with $n=5000$. In the last column we show the theoretical true values. The algorithm is applied with $\ell=8$.}
	\label{coef_sq5000}
\end{table}

\begin{table}[hbt!]
	\footnotesize
	\centering
	\begin{tabular}{|ccc|ccc|ccc|l}
		\multicolumn{3}{c|}{gr1} & \multicolumn{3}{c|}{gr2} & \multicolumn{3}{c|}{gr3} &true  values \\
		Mean   &  S.d.  &  MAD   &  Mean   &  S.d.   &  MAD   &  Mean   &  S.d.   &  MAD   &        \\
		7.7651  & 0.4452 & 0.3571 & 7.7539  & 0.4485 & 0.3597 & 7.7704  & 0.4394 & 0.3490 & 8      \\
		16.2247 & 2.2649 & 1.8021 & 16.2884 & 2.2669 & 1.8081 & 16.1809 & 2.0020 & 1.5940 & 16     \\
		5.8222  & 3.1254 & 2.4874 & 5.7306  & 3.0447 & 2.4010 & 5.9165  & 2.3929 & 1.9001 & 6.2832
	\end{tabular}
	\caption{\footnotesize For the union of squares: mean, standard deviation (S.d.) and MAD of the estimated polynomial coefficients, over 1000 replications, for $n=7000$. In the last column we show the theoretical true values. The algorithm is applied with $\ell=8$.}
	\label{coef_sq7000}
\end{table}

\begin{table}[hbt!]
	\footnotesize
	\centering
	\begin{tabular}{|ccc|ccc|ccc|l}
		\multicolumn{3}{c|}{gr1}  & \multicolumn{3}{c|}{gr2}  & \multicolumn{3}{c|}{gr3} &true values  \\
		Mean   &   S.d. &  MAD   &  mean   &   S.d.  &  MAD   &  Mean   &  S.d.  &  MAD   &         \\
		2.8962 & 0.2497 & 0.1943 & 2.9062 & 0.2562 & 0.2032 & 2.9096 & 0.2197 & 0.1756 & 3 \\
		12.3068 & 2.0234 & 1.4270 & 12.1894 & 2.1690 & 1.6791 & 12.1477 & 1.4384 & 1.1379 & 12 \\
		-1.4260 & 4.1637 & 2.5378 & -1.1316 & 4.5367 & 3.3791 & -1.0473 & 2.2401 & 1.7620 & -0.8584 \\
	\end{tabular}
	\caption{\footnotesize Mean, standard deviation (S.d.) and median absolute deviation (MAD) of the estimated polynomial coefficients for $S_F$, over 1000 replications, with $n=5000$. In the last column we show the theoretical true values. The algorithm is applied with $\ell=30$.}
	\label{marco5000}
\end{table}

\begin{table}[hbt!]
	\footnotesize
	\centering
	\begin{tabular}{|ccc|ccc|ccc|l}
		\multicolumn{3}{c|}{gr1} & \multicolumn{3}{c|}{gr2} & \multicolumn{3}{c|}{gr3} & true values \\
		 Mean   &  S.d.  &  MAD   &  Mean   &  S.d.  &  MAD   &  Mean   &   S.d.  &  MAD   &             \\
		2.9329  & 0.2098 & 0.1661 & 2.9201  & 0.2197 & 0.1750 & 2.9350  & 0.1987 & 0.1585 & 3           \\
		12.1758 & 1.5169 & 1.1275 & 12.2945 & 1.7744 & 1.3836 & 12.1525 & 1.2276 & 0.9920 & 12          \\
		-1.1706 & 2.9214 & 1.9174 & -1.4192 & 3.6725 & 2.7676 & -1.1160 & 1.9110 & 1.5321 & -0.8584
	\end{tabular}
	\caption{\footnotesize Mean, standard deviation (S.d.) and median absolute deviation (MAD) of the estimated polynomial coefficients  for $S_F$, over 1000 replications, with $n=7000$. In the last column we show the theoretical true values. The algorithm is applied with $\ell=30$.}
	\label{marco7000}
\end{table}

%


\end{document}